\documentclass{article}
\usepackage{amsthm,amsfonts, amsbsy, amssymb,amsmath,graphicx}
\usepackage{graphics}

\newtheorem{thm}{Theorem}

\newtheorem{ex}{Example}
\newtheorem{crl}{Corollary}
\newtheorem{dfn}{Definition}
\newtheorem{st}{Statement}
\newtheorem{rk}{Remark}
\newtheorem{prop}{Proposition}

\title{Free Knots and Parity}

\author{Vassily Olegovich Manturov \footnote{Partially
supported by RFBR, No. 07--01--00648}}

 \def\R{{\mathbb R}}
 \def\Z{{\mathbb Z}}
 \def\N{{\mathbb N}}
 \newcommand{\ZG}{\Z_{2}{\mathfrak{G}}}
 
 \def\F{{\cal F}}
 
\newcommand{\ZGG}{\tilde \Z_{2}{\mathfrak{G}}}

\newcommand{\skcrossr}{\raisebox{-0.25\height}{\includegraphics[width=0.5cm]{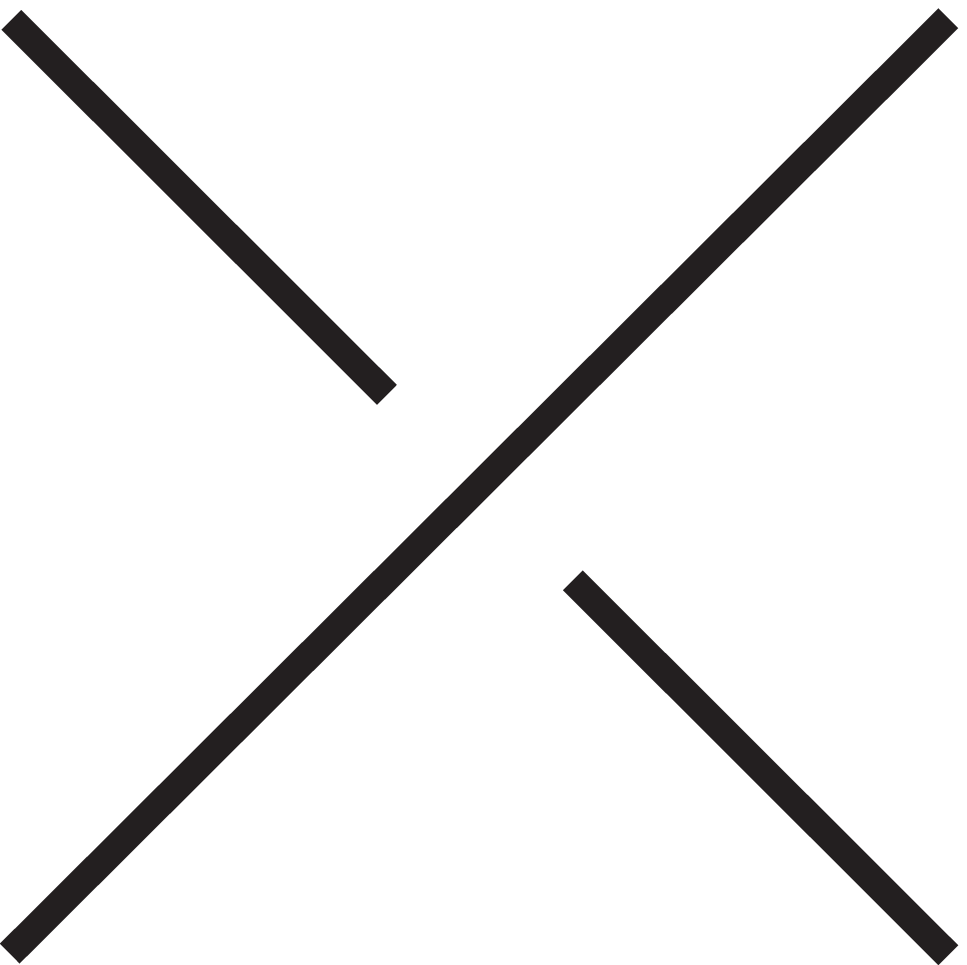}}}
\newcommand{\skcrv}{\raisebox{-0.25\height}{\includegraphics[width=0.5cm]{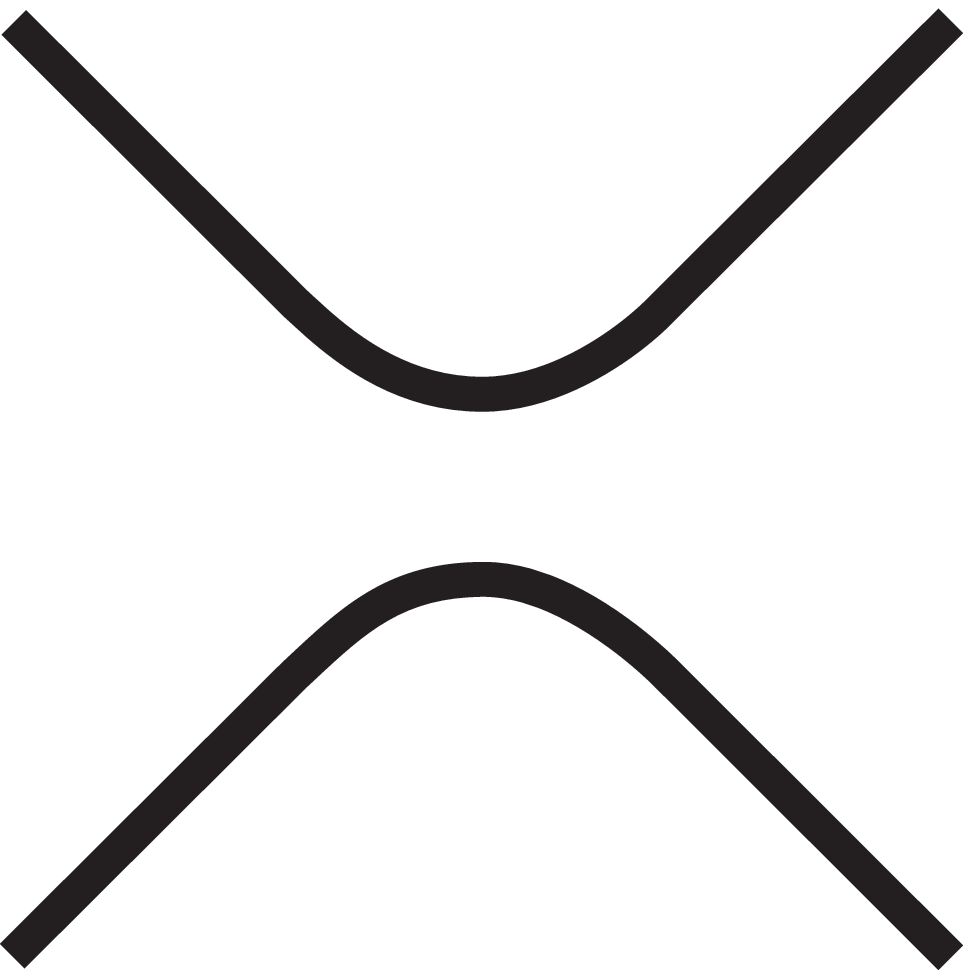}}}
\newcommand{\skcrh}{\raisebox{-0.25\height}{\includegraphics[width=0.5cm]{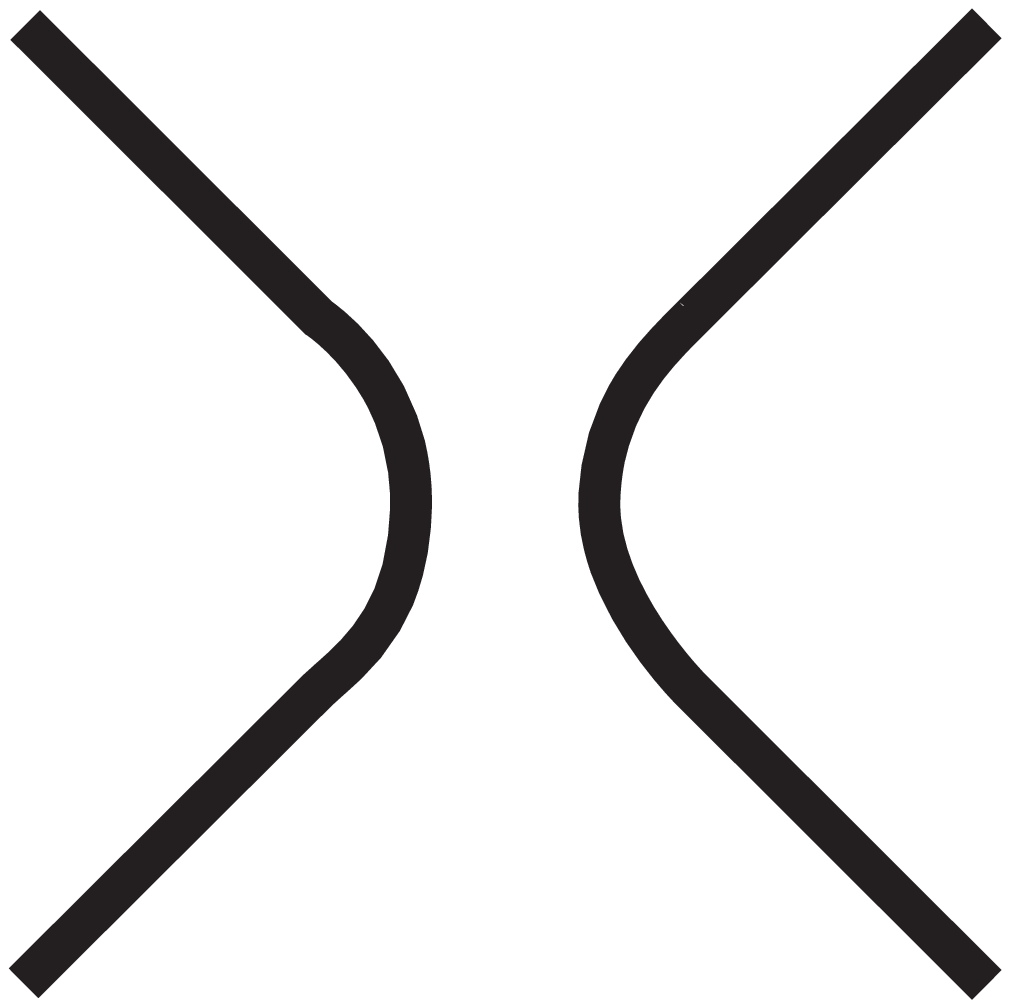}}}

\newcommand{\skcrro}{\raisebox{-0.25\height}{\includegraphics[width=0.5cm]{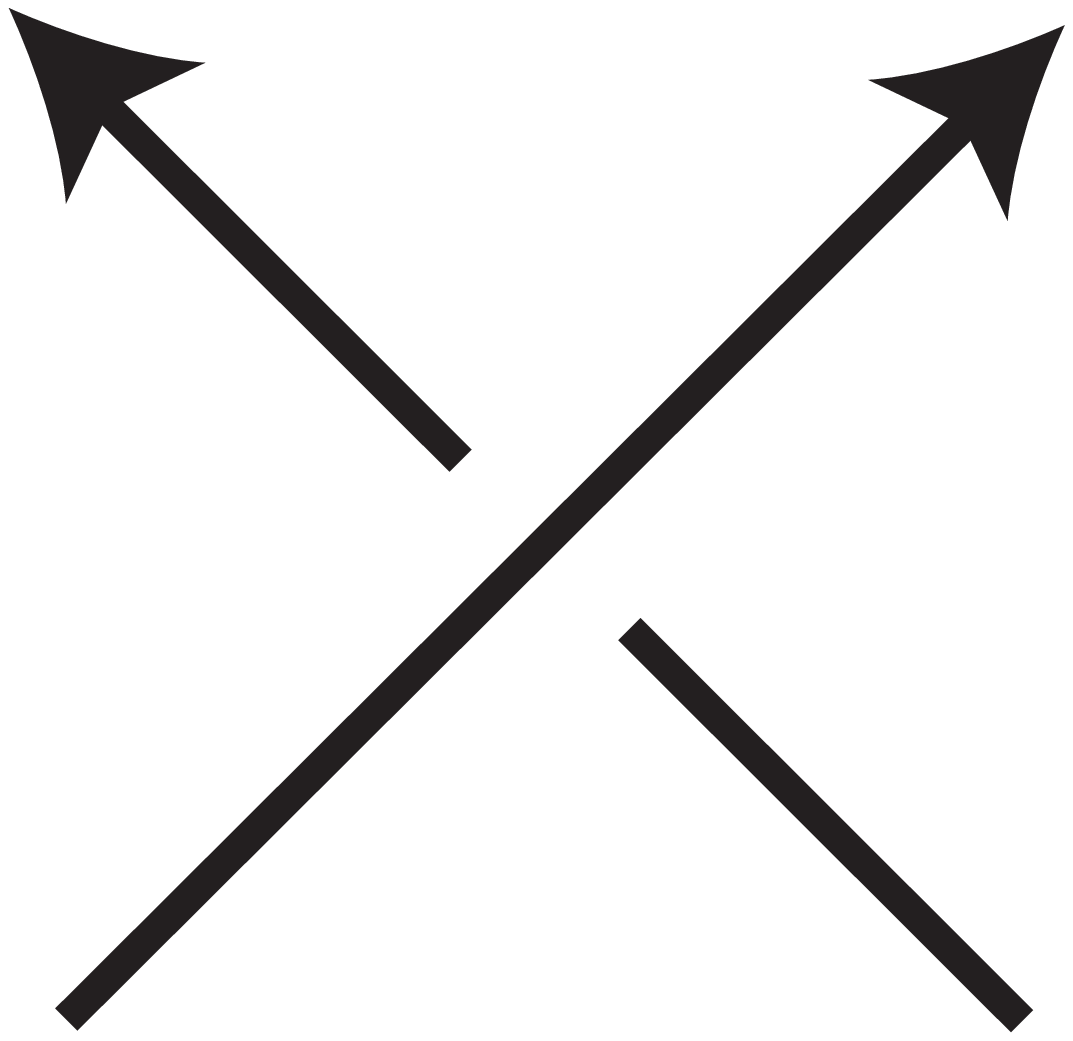}}}
\newcommand{\skcrlo}{\raisebox{-0.25\height}{\includegraphics[width=0.5cm]{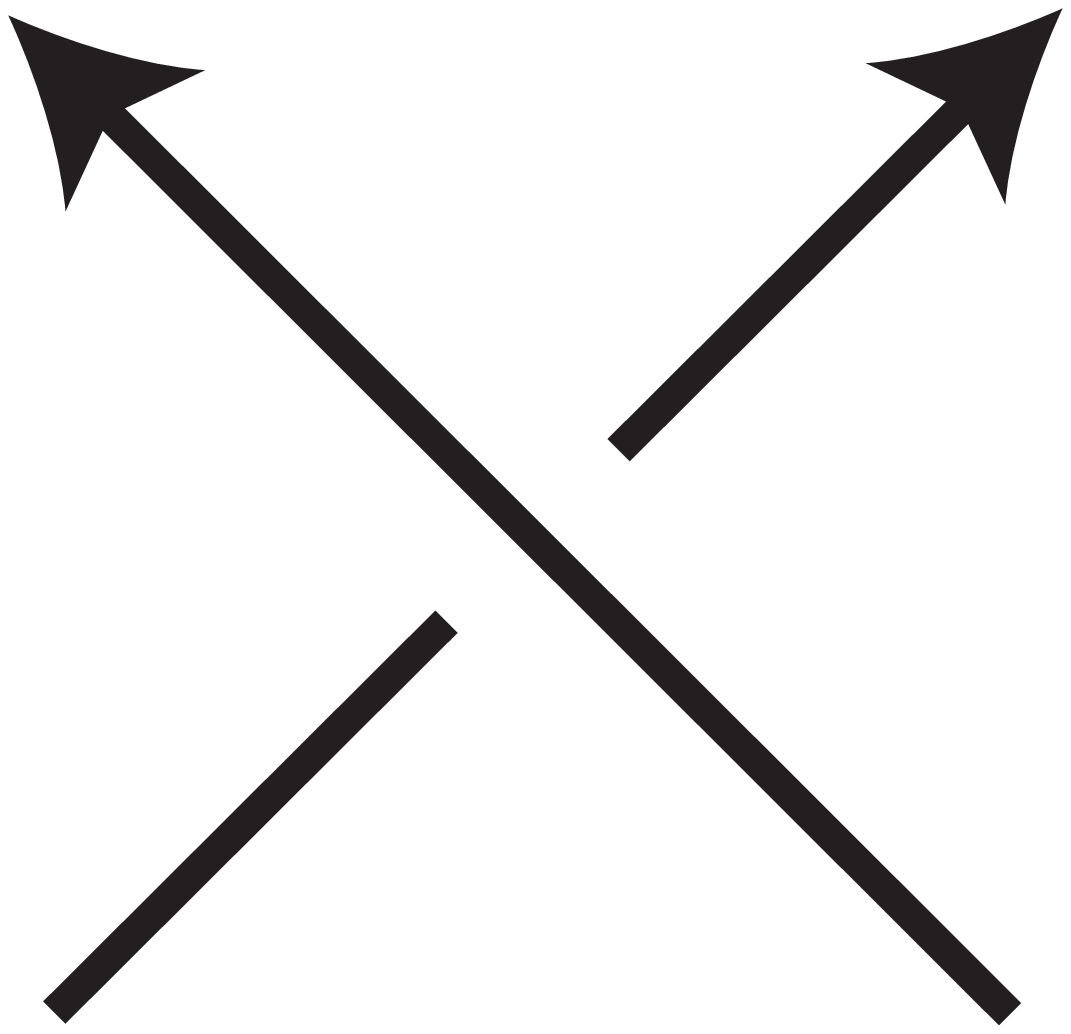}}}

\begin{document}

\maketitle

\abstract{We consider knot theories possessing a {\em parity}: each
crossing is decreed {\em odd} or {\em even} according to some
universal rule. If this rule satisfies some simple axioms concerning
the behaviour under Reidemeister moves, this leads to a possibility
of constructing new invariants and proving minimality and
non-triviality theorems for knots from these classes, and
constructing maps from knots to knots.

Our main example is virtual knot theory and its simplifaction, {\em
free knot theory}. By using Gauss diagrams, we show the existence of
non-trivial free knots (counterexample to Turaev's conjecture), and
construct simple and deep invariants made out of parity. Some
invariants are valued in graph-like objects and some other are
valued in groups.

We discuss applications of parity to virtual knots and ways of
extending well-known invariants.

The existence of a non-trivial parity for classical knots remains an
open problem. }


\section{Introduction}

The present paper describes the state-of-art of the theory of free
knots and links, a drastic simplification of virtual knot theory.
The main tool in our investigation is the concept of {\em parity}.
If there is a possibility to distinguish between two types of
crossings, {\em even ones} and {\em odd ones} in some knot theory,
this will allow us to construct new and powerful invariants and
improve some well-known invariants. In particular, parity was first
used to prove that free knots are in general non-trivial.

The first example of parity, called {\em Gaussian}, comes from the
chord diagram of a link: a chord (and, thus, the crossing
corresponding to it) is even whenever the number of chords linked
with it, is even. Summing up the properties of the Gaussian parity,
one can axiomatize the notion of parity for crossings. Most of
theorems we present here, work for {\em different} parities in
different situations.

Here we would like to emphasize the point of view that if there is a
possibility to distinguish between {\em different types} (not
necessarily two types) of crossings when looking at a diagram, this
information can be used to {\em improve many known invariants},
\cite{Ma,Ma1,Ma2,Af}. For rigorous proofs of most of the theorems
given here, see \cite{Ma,Ma1,Ma2}.

The concepts of the present paper have some far-reaching
generalizations in other situations. In particular, parity is very
helpful in the theory of graph-links, see review \cite{IM} in the
present volume.

The present paper is organized as follows. In the next section we
describe the notions of virtual knots and free knots.

Section 3 is devoted to the first examples and applications of
parity; we construct ``forgetting mappings'' and reprove the theorem
on ``closedness'' of the class of even diagrams.

Then in section 4 we redefine the concept of parity in terms of
homology of graphs.

Section 5 deals with invariants of free knots and virtual knots and
minimality results.

We conclude the section by a very simply constructed invariant of
free knots valued in a certain group, see \cite{MM}; this invariant
gives an obstruction for a free knot to be slice.

\subsection{Acknowledgement}

I am grateful to O.V.Manturov, D.P.Ilyutko for useful comments
concerning this work.

\section{Virtual knots. Free Knots}

Virtual knots were first introduced by Kauffman \cite{KaV} as an
extension of classical knots into knots in thickened surfaces
considered up to isotopy and stabilizations/destabilizations.
Virtual knots have a simple diagrammatic descriptions by planar
diagrams and Reidemeister moves. Another description of virtual
knots comes from Gauss diagrams \cite{GPV}: one admits all possible
Gauss diagrams (not necessarily realizing planar Gauss codes) and
factorizes them by formal Reidemeister moves.

A {\em virtual diagram} is a generic immersion of a framed $4$-graph
in â $\R^2$ with each vertex endowed with a classical crossing
structure, i.e. one pair of opposite half edges is defined to form
an overcrossing; the other two half-edges form an undercrossing;
when drawing a diagram on the plane, intersections of different
edges of the framed $4$-graphs in interior points are called {\em
virtual crossings}, and they are encircled.

By abusing notation, we shall allow a $4$-graph (four-valent graph)
to have connected free components homeomorphic to a circle; this
agrees with the fact that an unknot may have a crossingless diagram.
We shall denote the one-component crossingless diagram by
$\Gamma_{0}$.

A {\em virtual link} is an equivalence class of virtual diagrams by
usual Reidemeister moves applied to classical crossings and the {\em
detour move}. The detour move changes the immersion of an edge of
the graph: one takes an  edge fragment drawn on the plane (which has
only virtual crossings and self-crossings inside) and redraws it
arbitrarily in a generic way with all new crossings specified as
virtual.

So, every virtual knot or link has a ``frame'' being a framed
$4$-graph, whence virtual crossings ``are not there''. As shown in
\cite{GPV}, classical links embed into virtual links, i.e., {\em
every  two classical knots which are equivalent as virtual knots are
equivalent as classical knots (isotopic)}.

In fact, embedding theorems of such sort saying that the equivalence
relation on a subset coincides with the induced equivalence relation
coming from a set, can be proved by constructing projection ``maps''
from the whole set to a subset. The concept of parity can be used in
this direction in many cases, see Theorem \ref{ManturovViro} ahead.

We are now going to turn to a simplification of virtual knots
obtained by ``forgetting'' all structures except framing at
classical crossings.

{\em Free knots and links} (also known as homotopy classes of Gauss
words, resp., Gauss phrases, following Turaev \cite{Tu}) are defined
as follows.

By a {\em framed 4-graph} we mean a $4$-valent graph where for each
vertex we fix a way of splitting of the four emanating half-edges
into two pairs of edges called {\em (formally) opposite}.

Such graphs naturally arise from knot (or virtual knot) projections
or from generically immersed curves in $2$-surfaces.

A {\em free link} is an equivalence class of framed $4$-graphs
modulo the following {\em Reidemeister moves}:

1) addition/removal of a loop formed by two half-edges of the same
edge approaching a vertex from non-opposite sides;

2) addition/removal of a bigon with two edges being non-opposite at
both vertices;

3) the triangular move pulling a piece of unicursal curve through a
crossing formed by two other pieces of unicursal curves, see
Fig.\ref{sootvt}.

The  relation of half-edges to be opposite allows one to define the
notion of {\em component} and to count the number of {\em unicursal
components} of a framed graph, which turns out to be invariant under
Reidemeister moves. A {\em free knot} is a single-component free
link.

There are obvious ``forgetting maps'' from virtual links to free
links and from homotopy classes of curves on $2$-surfaces to free
links; classical knots (or planar curves) project to free unknots.

Thus, the study of invariants of free links may enlarge our
knowledge about virtual knots or knots in surfaces.

For two-component free links there is an obvious invariant: the
parity of number of the intersection points; for an $n$-component
free links there is a collection of mutual parities which can be
encoded by a graph on $n$ vertices.

Several years ago, V.G. Turaev \cite{TuProblem} conjectured that all
{\em free knots} are trivial.

This conjecture was recently disproved  by myself and
(independently, just some days later) A.Gibson \cite{Gib}. To do
that, I used the concept of {\em parity}.

The existence of a non-trivial parity for classical knots remains an
open problem. For different examples of parity, see \cite{Ma}.

\section{The concept of parity}

We shall encode knots and virtual knots \cite{KaV} by Gauss diagrams
\cite{GPV}. We say that two chords of a chord diagram are {\em
linked} if both ends of one of them lie in different connected
components of the two-component set obtained from the circle by
deleting the ends of the second components. We say that a chord $c$
of a Gauss diagram is {\em even} if the number of chords it is
linked with, is even, and {\em odd}, otherwise. We shall denote the
set of chords linked with $c$ by $E_{c}$; any chord is assumed not
to be linked with itself; by $E_{a}+E_{b}$ we shall mean $E_{a}\cup
E_{b}\backslash (E_{a}\cap E_{b})$. For two chords $a,b$, we write
$\langle a,b\rangle=1$ if $a$ and $b$ are linked, and $\langle
a,b\rangle=0$, otherwise. Here $+$ stays for the Boolean sum of
stes.

We say that a vertex of the  four-valent framed graph corresponding
to the Gauss diagram is {\em odd} whenever the chord corresponding
to this vertex is odd.

Let us summarize the properties of {\em even and odd chords} with
respect to their behaviour under Reidemeister moves. It is easy to
see that

\begin{enumerate}

\item The parity of a vertex participating in the first Reidemeister
move, is even.

The first Reidemeister move corresponds to an addition/removal of an
even vertex, and the pairwise incidence relation of the remanining
vertices does not change.

\item The second Reidemeister move adds (removes) two vertices $a,b$ of
the same parity.

Indeed, for such vertices $a,b$ herewith  $E_{a}+E_{b}$ is either
$\emptyset$ or coincides with $a+b$; after applying a Reidemeister
move the parity of the remaining vertices does not change. Neither
does the pairwise incidence of the remaining vertices.

\item The sum of parities (modulo 2) of the three vertices participating in a
third Reidemeister move is even. The parity of a vertex
participating in a third Reidemeister move does not change after
this move is applied.

Indeed, when performing the third Reidemeister move, we have three
vertices (chords) $a,b,c$, for which  $E_{a}+E_{b}+E_{c}\subset
\{a,b,c\}, |E_{a}+E_{b}+E_{c}|=0$ or $2$. After performing the move,
we get instead of $a,b,c$ three vertices $a',b',c'$ with pairwise
switched incidences (w.r.t. $a,b,c$) the remaining incidences are
unchanged: for $f,g\notin \{a,b,c\}: \langle f,g\rangle$ remains
unchanged:

a) $\langle f,a\rangle=\langle f,a'\rangle;\langle
f,b\rangle=\langle f,b'\rangle;\langle f,c\rangle=\langle f,
c'\rangle$ and

b) the number of {\em odd} vertices amongst $a,b,c$ is {\em even}
(is equal to zero or two).

\end{enumerate}

We call the parity of vertices of the Gauss diagram (and
corresponding chords)  the {\em Gaussian parity}.

\begin{dfn}
By a {\em parity} we shall mean a function on the set of all
four-valent framed graphs which satisfies the conditions above
(assuming all vertices not taking part in a Reidemeister move
preserve their parity).
\end{dfn}

\begin{rk}
Here, the vertices of a 4-valent graph are not enumerated, so the
parity function on should be symmetric with respect to the action of
the group of symmetries of the graph respecting the framing.
\end{rk}

In the case when we deal not merely with free knots or links, but
with knots having some decorations, by {\em parity} we shall mean
the property of crossings satisfying the same axioms with respect to
the Reidemeister moves. Thus, for instance, the second parity axiom
for the second Reidemeister move requires that the parity is the
same for both vertices of any ``bigon'' in the case of free knots,
whence for virtual knots, we shall require this only for those pairs
of crossings which can participate in a second Reidemeister move,
that is, for two crossings {\em of opposite signs}
($\skcrlo,\skcrro$) forming a bigon.

Certainly, each parity defined for free links induces a parity for
virtual knots, but not vice versa.

In the sequel, we shall state some theorems using parity. Unless
specified otherwise, such theorems will hold {\em for any parity}.

The existence of non-trivial parity for classical knots remains an
open problem.

\subsection{The functorial map $f$}

\label{FMap}

It turns out that the parity axioms listed above lead to a simple
and powerful map on the set of free knots and, more generally, on
the set of virtual knots.

Let $K$ be a virtual knot diagram. Let $f$ be a diagram obtained
from $K$ by making all {\em odd} crossings virtual. In other words,
we remove all odd chords.

The following theorem follows from definitions.

\begin{thm}
The map $f$ is a well-defined map on the set of all virtual knots.
For a virtual knot diagram $K$, $f(K)=K$ iff all crossings of $K$
are even. Otherwise, the number of classical crossings of $f(K)$ is
strictly less than the number of classical crossings of
$K$.\label{vyro}
\end{thm}

\subsection{Other examples of parities. The parity hierarchy}

The Gaussian parity (defined via intersection as above) is not the
only parity for knots. We shall give two more examples which can be
applied in different situation.

The first example deals with two-component free links. A crossing of
a $2$-component free link is called {\em even} if it is formed by
two branches of the same component; otherwise it is called {\em
odd}.

Another example deals with the set of virtual crossings being even.
Namely, consider the set of virtual knots with all crossings having
{\em even} Gaussian parity. It follows from Theorem \ref{vyro} that
this set is closed with respect to the Reidemeister moves, i.e.,
whenever two diagrams having even crossings are equivalent, they are
equivalent by a sequence of Reidemeister moves where all
intermediate diagrams have all even crossings. Denote this class of
knots by ${\mathfrak{V}^{1}}$; evidently, it includes all classical
knots. We shall show that there exists a natural {\em filtration} on
the set of all virtual knots:

$${\mathfrak{V}}^{0}\supset{\mathfrak{V}}^{1}\supset
{\mathfrak{V}}^{2}\supset \cdots \supset{\mathfrak{V}}^{n}\supset
\cdots,$$ which starts with the set of ${\mathfrak{V}^{0}}$ of all
virtual knots and has as a limit some set ${\mathfrak{V}^{\infty}}$
of ``knots of index zero'' including all classical knots.

So, let  $D$ be a virtual diagram, and let $G(D)$ be the
corresponding Gauss diagram. We shall endow the diagram $G$ with
signs and arrows in the usual way (the sign plus corresponds to
$\skcrro$ and the sign minus corresponds to $\skcrlo$; an arrow is
directed from the pre-image of the arc forming an overpass to the
pre-image of the arc forming an underpass).

With each classical crossing we associate its {\em index} which will
be a non-negative integer. Let $X$ be a classical crossing of $D$,
and let $c(X)$ be the corresponding (oriented) chord of $G(D)$.
Consider all chords of the diagram  $G(D)$, which are linked with
$c(X)$. Let us calculate the sum of signs of all those chords
intersecting $c(X)$ form the left to the right and subtract the sum
of signs of those chords  intersecting $c(X)$ from the right to the
left.

The {\em index of $c$} is the absolute value of the above quantity;
we shall denote it by $ind(X)\equiv ind(c(X))$.

Clearly, a chord is even whenever its index is even.

Let us collect some facts concerning the index; the proof can be
obtained by a simple check:

\begin{st}
\begin{enumerate}

\item Any chord taking part in a Reidemeister-1 move has index $0$.

\item In the Reidemeister-2 move indices of the two chords are the same.

\item Like parity, the index is invariant under the third Reidemeister move
$ind(a)=ind(a'),ind(b)=ind(b'),ind(c')$; moreover, if some three
chords $a,b,c$ take part in a third Reidemeister move then
$ind(a)\pm ind(b)\pm ind(c)=0$, see Fig. \ref{sootvt}.

\item The index remains unchanged under any Reidemeister move for those
chords not participating in this move.

\item All crossings of a classical diagram have index zero.

\end{enumerate}\label{st1s}
\end{st}

\begin{figure}
\centering\includegraphics[width=200pt]{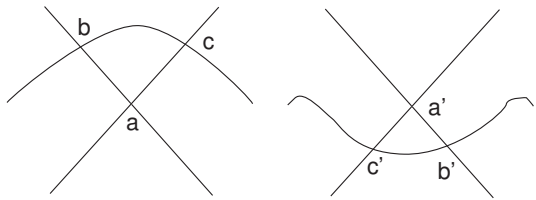}
\caption{Correspondence of crossings for the third Reidemeister
move} \label{sootvt}
\end{figure}

The properties described above show that the index can be used for
defining parities.

Indeed, from Statement \ref{st1s} we see that for diagrams from
${\mathfrak V^{1}}$ we may introduce the following parity: let $K$
be a knot diagram from ${\mathfrak V^{1}}$; we decree those
crossings of $K$ having index divisible by four, to be even, and
those having index congruent to two modulo four, to be odd.

From Statement \ref{st1s}, it follows that the parity defined in
this way, satisfies all the axioms.

Applying  $f$ to knots from ${\mathfrak V^{1}}$ with respect to the
parity described above (Gaussian parity) we get the set  ${\mathfrak
V^{2}}$ of diagrams with all indices divisible by four.

Arguing as above, we define the sets ${\mathfrak V^{k}}$ of diagrams
with indices divisible by $2^{k}, k\in\N$. Let ${\mathfrak
V^{\infty}}$ be the set of diagrams with all crossings having index
zero.

The following theorem holds
\begin{thm}
Let $K,K'$ be two diagrams of virtual knots from ${\mathfrak V^{k}}$
(where $k$ is a positive integer or  $\infty$), representing
equivalent virtual knots. Then there is a chain of diagrams
$K=K_{0}\to K_{1} \dots \to K_{n}=K'$ from ${\mathfrak V^{k}}$ where
every two adjacent diagrams differ by a Reidemeister move (or a
detour move).\label{index}
\end{thm}

The proof of Theorem \ref{index} is obtained as follows. Starting
from an arbitrary chain connecting $K$ to $K'$, we consequently
apply the map $f$ for the Gaussian parity until the chain belongs to
${\mathfrak V^{1}}$, then we consequently apply the other $f$ to get
a sequence from ${\mathfrak V^{2}}$, and so on, until we get a chain
from ${\mathfrak V^{k}}$. Since all of our maps $f$ (applied to
different ${\mathfrak V^{k}}$) are well defined, each two
consecutive elements in each of our chains will be connected by a
Reidemeister move or a detour move.

 In the case $k=\infty$ it is
sufficient to apply the trick from Theorem \ref{ManturovViro}
finitely many times, also. Indeed, consider an arbitrary chain
connecting $K$ to $K'$. Assume the maximal number of chords of
diagrams in this chain does not exceed $m<2^{m}$. If an index of
some chord is divisible by $2^{m}$ then this index is equal to zero.
So, it will be sufficient to apply the map $f$ $m$ times to get a
sequence of diagrams from ${\mathfrak V}^{\infty}$.

Note that the class ${\mathfrak V^{\infty}}$ is quite interesting:
it is an ``approximation'' of classical knots in the set of virtual
knots. All invariants defined on the set ${\mathfrak V^{\infty}}$,
can be translated to virtual knots by using $f$ (in virtue of
Theorem \ref{index}). Obviously, all classical diagrams are diagrams
belong to ${\mathfrak V}^{\infty}$. A non-classical knot diagram $D$
from ${\mathfrak V^{\infty}}$ is given in Fig. \ref{neklassich}.

\begin{figure}
\centering\includegraphics[width=150pt]{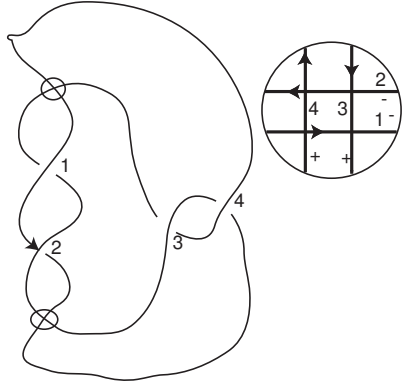} \caption{A
non-classical diagram from ${\mathfrak V^{\infty}}$}
\label{neklassich}
 \end{figure}

\section{Parity as homology}

Consider a framed $4$-graph $\Gamma$ with one unicursal component.
The homology group $H_{1}(\Gamma,\Z_{2})$ is generated by ``halves''
corresponding to vertices: for every vertex $v$ we have two halves
of the graph  $\Gamma_{v,1}$ and $\Gamma_{v,2}$, obtained by
smoothing at this vertex, see Fig. \ref{smo}. If the set of framed
$4$-graphs (possibly, with some further decorations at vertices) is
endowed with a {\em parity}, we may assume that we are given the
following cohomology class $h$: for each of the halves
$\Gamma_{v,1},\Gamma_{v,2}$ we set
$h(\Gamma_{v,1})=h(\Gamma_{v,2})=p(v)$, where $p(v)$ is the parity
of the vertex $v$. Taking into account that every two halves sum up
to give the cycle generated by the whole graph, we have defined a
``characteristic'' cohomology class $h$ from $H_{1}(\Gamma,\Z_{2})$.

\begin{figure}
\centering\includegraphics[width=180pt]{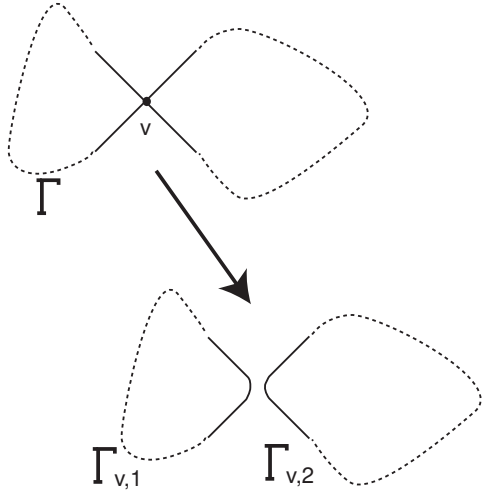} \caption{The graphs
$\Gamma_{v,1}$ and $\Gamma_{v,2}$} \label{smo}
\end{figure}

Collecting the properties of this cohomology class we see that

\begin{enumerate}

\item For every framed $4$-graph $\Gamma$ we have $h(\Gamma)=0$.

\item If $\Gamma'$ is obtained from  $\Gamma$ by a first
Reidemeister move adding a loop then for every basis
$\{\alpha_{i}\}$ of $H_{1}(\Gamma,\Z_{2})$ there exists a basis of
the group $H_{1}(\Gamma,\Z_{2})$ consisting of one element $\beta$
corresponding to the loop and a set of elements $\alpha'_{i}$
naturally corresponding to $\alpha_{i}$.

Then we have $h(\beta)=0$ and $h(\alpha_{i})=h(\alpha'_{i})$.

\item Let $\Gamma'$ be obtained from $\Gamma$ by a third increasing
Reidemeister move. Then for every basis $\{\alpha_{i}\}$ of
 $H_{1}(\Gamma,\Z_{2})$ there exists a basis in
$H_{1}(\Gamma',\Z_{2})$ consisting of one ``bigon'' $\gamma$, the
elements $\alpha'_{i}$ naturally corresponding to $\alpha_{i}$ and
one additional element $\delta$, see Fig. \ref{r2r3}, left.

\begin{figure}
\centering\includegraphics[width=220pt]{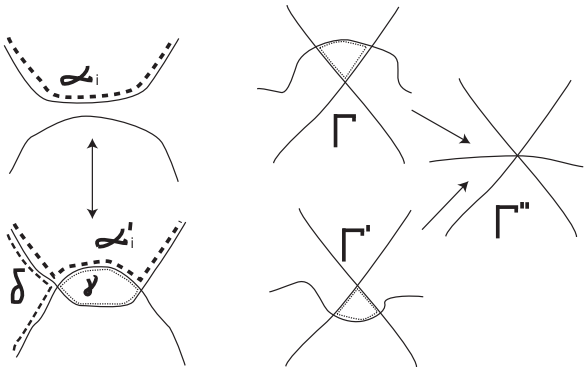} \caption{The
cohomology condition for Reidemeister moves} \label{r2r3}
\end{figure}

Then the following holds: $h(\alpha_{i})=h(\alpha'_{i})$,
$h(\gamma)=0$.

\item Let $\Gamma'$ be obtained from $\Gamma$ by a third
Reidemeister move. Then there exists a graph $\Gamma''$ with one
vertex of valency  $6$ and the other vertices of valency $4$ which
is obtained from either of $\Gamma$ or $\Gamma'$ by contracting the
``small'' triangle to the point. This generates the mappings
$i:H_{1}(\Gamma,\Z_{2})\to H_{1}(\Gamma'',\Z_{2})$ and
$i':H_{1}(\Gamma',\Z_{2})\to H_{1}(\Gamma'',\Z_{2})$, see Fig.
\ref{r2r3}, right.

Then the following holds: the cocycle $h$ is equal to zero for small
triangles, besides that if  for $a\in H_{1}(\Gamma,\Z_{2}),a'\in
H_{1}(\Gamma',\Z_{2})$ we have $i(a)=i'(a')$, then $h(a)=h(a')$.

\end{enumerate}

Thus, every parity for free knots generates some $\Z_{2}$-cohomology
class for all framed $4$-graphs with one unicursal component, and
this class behaves nicely under Reidemeister moves.

The converse is true as well. Assume we are given a certain
``universal'' $\Z_{2}$-cohomology class for all four-valent framed
graphs satisfying the conditions 1)-4) described above. Then it
originates from some {\em parity}. Indeed, it is sufficient to
define the parity of every vertex to be the parity of the
corresponding half. The choice of a particular half does not matter,
since the value of the cohomology class on the whole graph is zero.
One can easily check that parity axioms follow.

This point of view allows to find parities for those knots lying in
$\Z_{2}$-homologically nontrivial manifolds. For more details, see
\cite{Ma}.

\section{Invariants and minimality examples}

As an example showing the power of the notion of parity we present
the following theorem.

\begin{thm}
Let $K$ be a four-valent framed graph with one unicursal component
such that all chords of $K$ are odd and no second decreasing
Reidemeister move can be applied to $K$. Then $K$ is a minimal
diagram of the corresponding free link in the following strong
sense: for any diagram $K'$ equivalent to $K$ there is a smoothing
of $K'$ isomorphic to the graph $K$.\label{mnm}
\end{thm}

We shall call four-valent framed graphs from the formulation of
Theorem \ref{mnm} {\em irreducibly odd}.

In fact, we shall see that irreducibly odd diagrams are mininmal in
a stronger sense, see Corollary \ref{sld} ahead.

To prove this theorem, we shall introduce an invariant $[\cdot]$
valued in linear combinations of some graphs (more precisely,
equivalence classes of graphs), which allows to reduce all the
Reidemeister moves to the second Reidemeister move only.

Let ${\mathfrak{G}}$ be the set of all equivalence classes of framed
 graphs with one unicursal component modulo second Reidemeister moves. Consider the linear space $\ZG$.

Let $\Gamma$ be a framed graph, let  $v$ be a vertex of $\Gamma$
with four incident half-edges $a,b,c,d$, s.t.  $a$ is opposite to
$c$ and $b$ is opposite to $d$ at $v$.

By {\em smoothing} of $\Gamma$ at $v$ we mean any of the two framed
$4$-graphs obtained by removing $v$ and repasting the edges as
$a-b$, $c-d$ or as $a-d,b-c$, see Fig. \ref{smooth}.

\begin{figure}
\centering\includegraphics[width=200pt]{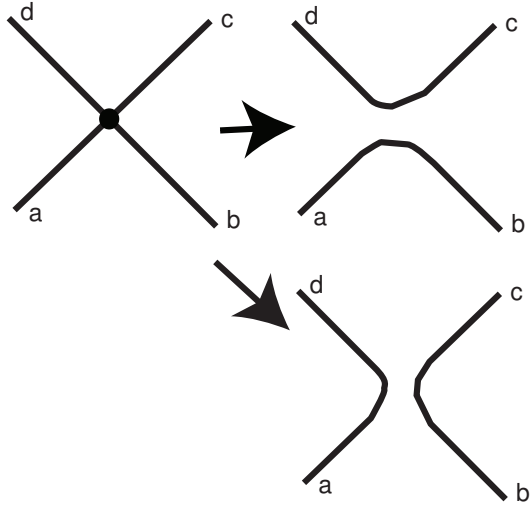} \caption{Two
smoothings of a vertex of for a framed graph} \label{smooth}
\end{figure}

Herewith, the rest of the graph (together with all framings at
vertices except $v$) remains unchanged.

We may then consider further smoothings of $\Gamma$ at {\em several}
vertices.

Consider the following sum

\begin{equation}
[\Gamma]=\sum_{s\;even.,1\; comp} \Gamma_{s},
\end{equation}
which is taken over all smoothings in all {\em even} vertices, and
only those summands are taken into account where $\Gamma_{s}$ has
one unicursal component.

Thus, if  $\Gamma$ has $k$ even vertices, then $[\Gamma]$ will
contain at most $2^{k}$ summands, and if all vertices of $\Gamma$
are odd, then we shall have exactly one summand, the graph $\Gamma$
itself.

Consider  $[\Gamma]$ as an element of $\ZG$. In this case it is
evident, for instance, that if all vertices of $\Gamma$ are even
then $[\Gamma]=[\Gamma_0]$: by construction, all summands in the
definition of $[\Gamma]$ are equal to $[\Gamma_0]$, it can be easily
checked that the number of such summands is odd.

Now, we are ready to formulate the main result of the present
section:

\begin{thm}
If $\Gamma$ and $\Gamma'$ represent the same free knot then in $\ZG$
the following equality holds: $[\Gamma]=[\Gamma']$.\label{mainthm}
\end{thm}

Theorem \ref{mainthm} yields the following
\begin{crl}
Let  $\Gamma$ be an irreducibly odd framed 4-graph with one
unicursal component. Then any representative $\Gamma'$ of the free
knot $K_{\Gamma}$, generated by $\Gamma$, has a smoothing  $\tilde
\Gamma$ having the same number of vertices as $\Gamma$. In
particular, $\Gamma$ is a minimal representative of the free knot
$K_{\Gamma}$ with respect to the number of vertices.\label{sld}
\end{crl}

\begin{proof}[Proof of the Corollary]
By definition of $[\Gamma]$ we have $[\Gamma]=\Gamma$. Thus if
$\Gamma'$ generates the same free knot as $\Gamma$ we have
$[\Gamma']=\Gamma$ in $\ZG$.

Consequently, the sum representing $[\Gamma']$ in $\mathfrak{G}$
contains at least one summand which is equivalent to  $\Gamma$ in
$\ZG$. Thus $\Gamma'$ has at least as many vertices as $\Gamma$
does.

Moreover, the corresponding smoothing of $\Gamma'$ is a diagram,
which is equivalent to $\Gamma$ by second Reidemeister moves. The
irreducibility condition yields that one of smoothings of $\Gamma'$
coincides with $\Gamma$.
\end{proof}

The invariant $[\cdot]$ is constructive. To see that, one should
just look at the structure of the set ${\ZG}$.

\begin{ex}
The simplest example of an irreducibly odd graph (which is minimal
according to Theorem \ref{mnm} is depicted in Fig. \ref{irred}.
\begin{figure}
\centering\includegraphics[width=150pt]{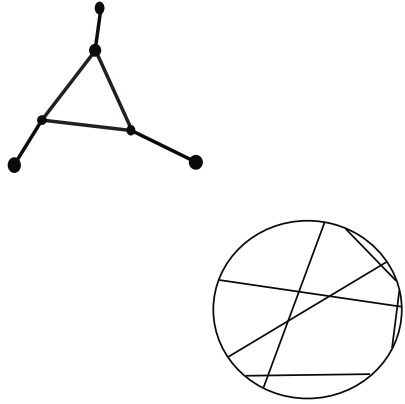} \caption{An
irreducibly odd graph and its chord diagram} \label{irred}
\end{figure}
\end{ex}
\begin{ex}
A free link whose minimality can be established analogously by using
$\{\cdot\}$ is shown in Fig.\ref{frlnk1}
\begin{figure}
\centering\includegraphics[width=200pt]{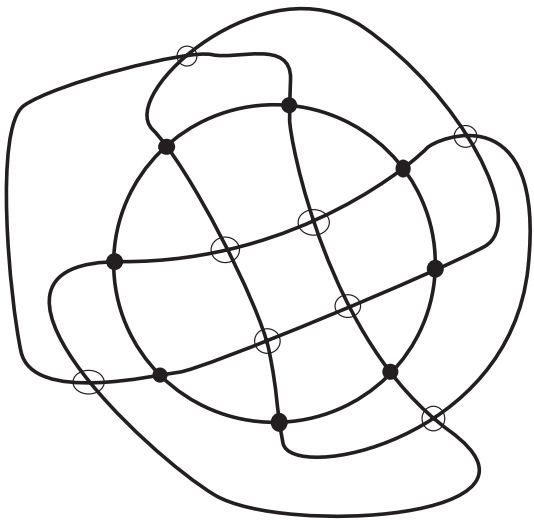}
\caption{Minimal free two-component link} \label{frlnk1}
\end{figure}
\end{ex}

\begin{rk}
In the case of free links with a parity, the bracket $[\cdot]$ does
not work verbatim: by definition, for any free link $L$ with all odd
crossings we have $[L]=0$ since the only even smoothing of $L$ gives
more than one component. Nevertheless, there is a simple
modification of $[\cdot]$, denoted by $\{\cdot\}$, which is the sum
over all even smoothings, which yield links with no split
components. $\{\cdot\}$ allows to prove minimality theorems for
links with arbitrarily many components, for more details, see
\cite{Ma}.
\end{rk}

\subsection{The set ${\ZG}$}

Having a framed $4$-graph, one can consider it as an element of
$\ZG$. It is natural to try simplifying it: we call a graph in $\ZG$
{\em irreducible} if no decreasing second Reidemeister move can be
applied to it. The following theorem is trivial

\begin{thm}
Every $4$-valent framed graph $G$ with one unicursal component
considered as an element of $\ZG$ has a unique irreducible
representative, which can be obtained from $G$ by consequtive
application of second decreasing Reidemeister moves.

\end{thm}

This allows to recognize elements $\ZG$  easily, which makes the
invariants constructed in the previous subsection digestable.

In particular, the minimality of a framed $4$-graph in $\ZG$  is
easily detectable: one should just check all pairs of vertices and
see whether any of them can be cancelled by a second Reidemeister
move (or in $\ZG$ one should also look for free loops) .

\subsection{The ``even'' Kauffman bracket}

We have shown that the parity consideration allows to construct new
invariants. In fact, parity considerations allow to improve some
existing invariants, as well. Below we show how to strengthen the
Kauffman bracket by using parity, for details see \cite{Ma}. A
construction of a generalized Alexander polynomial and a generalized
quandle are given in  D.M.Afanasiev's paper \cite{Af}.

We shall explicitly write down all formulae for virtual knot theory.
The generalization of the Kauffman bracket given below works also in
some other cases, when for every crossing, one can distinguish
between the two ways of smoothing, $A$ and $B$, so that these
smoothings ``respect the Reidemeister moves as in the classical
case''; also note that these generalizations work for the case of
graph-links \cite{IM,IM1,IM2}.

We recall the definition of the Kauffman bracket. For a virtual
diagram $K$, we set

\begin{equation}
\langle K\rangle =\sum_{s}
a^{\alpha(s)-\beta(s)}(-a^{2}-a^{-2})^{\gamma(s)-1},
\end{equation}
where $\alpha(s)$ and $\beta(s)$ are the numbers of smoothings of
type $A:\skcrossr\to \skcrh$ and of type $B:\skcrossr\to \skcrv$,
respectively, and $\gamma(s)$ denotes the number of circles
(unicursal components) in the state $s$ (after smoothing). Note that
in the definition of the Kauffman bracket we {\em smooth all
crossings}. Thus, the value of the Kauffman bracket is just a
Laurent polynomial in $a$.

Consider the free module $\F$ over $\Z[a,a^{-1}]$ generated by all
framed $4$-graphs.

Let ${\tilde \F}$ be the quotient of $\F$ modulo the following
relation

\begin{enumerate}

\item The second Reidemeister move,

\item The relation $L\sqcup \bigcirc = (-a^{2}-a^{-2})L$, where $L$ stays for any framed $4$-graph and
$L\sqcup\bigcirc$ denote the split sum of $L$ with a circle.
\end{enumerate}

The algorithmic recognizability of elements from ${\tilde \F}$ is
similar to that of $\ZG$.

Let us use the Gaussian parity for virtual knots, and let us
construct the {\em even Kauffman bracket} taking values in the
module $\F$, as follows.

Now, by {\em states} we shall mean those states considered by
smoothings at {\em even crossings}.

 We set

\begin{equation}
\langle K\rangle_{even} =\sum_{s_{{\small{even}}}}
a^{\alpha(s)-\beta(s)}K_{s},
\end{equation}
where $K_{s}$ is the free link obtained from the diagram $K$ by
smoothing with respect to the state $s$. Here $K_{s}$ is considered
as an element from ${\tilde \F}$.

It is easy to see that the invariant $[\cdot]$ (as well as
$\{\cdot\}$) is a specification of the bracket defined above.

Then the following theorem takes place

\begin{thm} The bracket  $\langle \cdot \rangle_{\mbox{even}}$
is invariant under $\Omega_{2},\Omega_{3}$ (and the detour move).
When applying $\Omega_{1}$ the bracket $\langle\cdot \rangle_{even}$
gets multiplied by  $(-a)^{\pm 3}$. The following normalization for
$\langle \cdot\rangle_{even}$ is invariant under all Reidemeister
moves: $X_{\mbox{even}}(K)=(-a)^{-3w(K)}\langle
K\rangle_{\mbox{even}}$, where $w(K)$ stays for the writhe number of
the oriented diagram $K$\label{th15}
\end{thm}

We call $X_{even}(K)$ the {\em even Jones polynomial} of the virtual
knot $K$.

\begin{rk}
The even Jones polynomial is a generalization of the invariant
$\{\cdot\}$ for any knot theory possessing a parity and a
distinguished rule for $A$- and $B$-type smoothing respecting the
Reidemeister moves.

Moreover, for the Gaussian parity, the polynomial
$X_{\mbox{even}}(K)$ is a generalization of the usual Jones
polynomial for classical knots and for knots with all crossings
being even. In this case, in the definition of $\langle \cdot
\rangle_{\mbox{even}}$, all elements $K_{s}$ are free links, which
in the module $\tilde \F$ are multiples of the unknot with
coefficients being powers of the polynomial $(-a^{2}-a^{-2})$.
Considering the generator of the module  ${\tilde \F}$,
corresponding to the unknot, as $1$, we get the standard Jones
polynomial.
\end{rk}

The invariance proof is analogous to the invariance proof of the
usual Kauffman bracket; for details, see \cite{Ma}.

\subsection{Atoms and Parity}

Theorem \ref{mnm} and Corollary \ref{sld} deal with only those free
knots having odd chords. However, the class of knots having with all
 even chords is very important. These are knots corresponding to
{\em orientable atoms}. We shall describe the notion of atom and its
connection to the Gaussian parity, and also construct an example of
a trivial link with all even chords.

In many situations, it is easier to find {\em links} rather than
{\em knots} with desired non-triviality properties. So, we shall
first define a map from free $1$-component links to $\Z_{2}$-linear
combinations of $2$-component links, and then we shall study the
latter by an invariant similar to that constructed in \cite{Ma}.

Ideologically, the first map is a simplified version of Turaev's
cobracket \cite{Turaev} which establishes a structure of Lie
coalgera on the set of curves immersed in $2$-surfaces (up to some
equivalence, the Lie {\em algebra} structure was introduced by
Goldman in a similar way). We shall use a simplification of Turaev's
construction adopted to the case of free knots and links.

The second map takes a certain state sum for a $2$-component free
link, where we distinguish between two types of crossings, and
smooth only crossings of the first type. What should these ``two
types'' mean, will be discussed later.

In some sense, the invariant $[\cdot]$ of free knots constructed in
\cite{Ma} is a diagrammatic extension of a terrifically simplified
Alexander polynomial (we forget about the variable and signs taking
$\Z_{2}$-coefficients). The invariant $\{\cdot\}$ suggested in the
present paper is in the same sense an extension of the terrifically
simplified Kauffman bracket, but again we use diagrams as
coefficients.

Altogether, these two constructions (the bracket and Turaev's
$\Delta$) provide an example of non-trivial and minimal diagrams of
free knots with orientable atoms.

\begin{dfn}
An {\em atom} (originally introduced by Fomenko, \cite{Fom}) is a
pair $(M,\Gamma)$ consisting of a $2$-manifold $M$ and a graph
$\Gamma$ embedded in $M$ together with a colouring of $M\backslash
\Gamma$ in a checkerboard manner. An atom is called {\em orientable}
if the surface $M$ is orientable. Here $\Gamma$ is called the {\em
frame} of the atom, whence by {\em genus} (atoms and their genera
were also studied by Turaev~\cite{Turg}, and atom genus is also
called the Turaev genus~\cite {Turg}) ({\em Euler characteristic,
orientation}) of the atom we mean that of the surface $M$.
 \end{dfn}

Having  an atom $V$, one can construct a virtual link diagram out of
it as follows. Take a generic immersion of atom's frame into $\R^2$,
for which the formally opposite structure of edges coincides with
the opposite structure induced from the plane.

Put virtual crossings at the intersection points of images of
different edges and restore classical crossings at images of
vertices `as above'. Obviously, since we disregard virtual
crossings, the most we can expect is the well-definiteness up to
detours. However, this allows us to get different virtual link types
from the same atom, since for every vertex $V$ of the atom with four
emanating half-edges $a,b,c,d$ (ordered cyclically on the atom) we
may get two different clockwise-orderings on the plane of embedding,
$(a,b,c,d)$ and $(a,d,c,b)$. This leads to a move called {\em
virtualisation}.

 \begin{dfn}
By a {\em virtualisation} of a classical crossing of a virtual
diagram we mean a local transformation shown in
Fig.~\ref{virtualisation}.
 \end{dfn}

 \begin{figure} \centering\includegraphics[width=100pt]{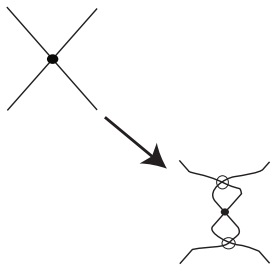}
  \caption{Virtualisation} \label{virtualisation}
 \end{figure}

The above statements summarise as

 \begin{prop}(see, e.g., {\em \cite{MyBook}}).
Let $L_1$ and $L_2$ be two virtual links obtained from the same atom
by using different immersions of its frame. Then $L_1$ differs from
$L_2$ by a sequence of detours and virtualisations.
 \end{prop}

At the level of Gauss diagrams, virtualisation is the move that does
not change the writhe numbers of crossings, but inverts the arrow
directions. So, atoms just keep the information about signs of Gauss
diagrams, but not of their arrows.

A further simplification comes when we want to forget about the
signs and pass to flat virtual links (see also \cite{VStrings}): in
this case we don't want to know which branch forms an overpass at a
classical crossing, and which one forms an underpass. So, the only
thing we should remember is its frame with opposite edge structure
of vertices (the {\em $A$-structure}). Having that, we take any atom
with this frame and restore a virtual knot up to virtualisation and
crossing change.

The $A$-structure of an atom's frame is exactly a $4$-valent framed
graph.

This perfectly agrees with the fact that {\bf free links are virtual
links modulo virtualization and crossing changes.}

Having a framed $4$-graph, one can consider {\em all atoms} which
can be obtained from it by attaching black and white cells to it. In
fact, it turns out that for a given framed $4$-graph either all such
surfaces are orientable or they are all non-orientable.

To see this, one should introduce the {\em source-sink} orientation.
By a {\em source-sink} orientation of a $4$-valent framed graph we
mean an orientation of all edges of this graph in such a way that
for each vertex some two opposite edges are outgoing, whence the
remaining two edges are incoming.

The following statement is left to the reader as an exercise
\begin{ex}
Let $G$ be a framed $4$-graph. Then the following conditions are
equivalent:

1. $G$ admits a source-sink orientation

2. At least one atom obtained from $G$ by attaching black and white
cells is orientable.

3. All atoms obtained from $G$ by attaching black and white cells
are orientable.

Moreover, if $G$ has one unicursal component, then each of the above
conditions is equivalent to the following: every chord of the
corresponding Gauss diagram $C(D)$ is even.

\end{ex}

Theorem \ref{vyro} leads to a new proof of a partial case the
following
\begin{thm}[First proved by O.Ya. Viro and V.O.Manturov, 2005, first
published in \cite{IM}] The set of virtual links with orientable
atoms is closed. In other words, if two virtual diagrams $K$ and
$K'$ have orientable atoms and they are equivalent, then there is a
sequence of diagrams $K=K_{0}\to K_{1}\dots\to K_{n}=K'$ all having
orientable atoms where $K_{i}$ is obtained from $K_{i}$ by a
Reidemeister move.\label{ManturovViro}
\end{thm}

Indeed, for the case of knots, this is a direct corollary from
Theorem \ref{vyro} in view of the exercise above; the general case
of this theorem can be proved by a modification of parity, called
{\em relative parity}, investigated by D.Yu.Krylov and the author,
\cite{KM}.

We give two examples: for a planar $4$-valent framed graph we
present a source-sink orientation (left picture, Fig. \ref{lfr}),
and for a non-orientable $4$-valent framed graph (right picture,
Fig.\ref{lfr}, the artefact of immersion is depicted by a virtual
crossing) we see that the source-sink orientation taken from the
left crossing leads to a contradiction for the right crossing.

\begin{figure}
\centering\includegraphics[width=200pt]{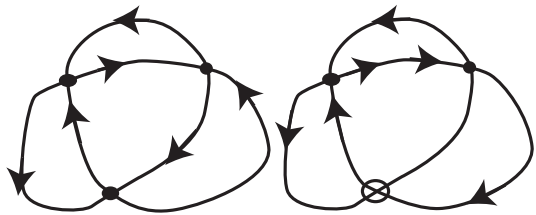} \caption{The
source-sink condition: satisfied (left); violated (right)}
\label{lfr}
\end{figure}

\subsection{Turaev's Cobracket and Its Even Analogue}

There is a simple and fertile idea due to Goldman \cite{Goldman} and
Turaev \cite{Turaev} of transforming two-component curves into
one-component curves and vice versa.

Here we simplify Turaev's idea for our purposes and we shall call
the map we are going to construct ``Turaev's $\Delta$''. Let
${\mathfrak{G}}_{l}$ is be the set of all equivalence classes of
framed graphs with $l$ unicursal components by the second
Reidemeister move and the relation that takes every framed $4$-graph
with a split component to $0$. We shall construct a map from $\ZG$
to $\ZG_{2}$ as follows.

In fact, to define the map $\Delta$, one may require for a free knot
to be oriented. However, we can do without.

Given a framed $4$-graph $G$. We shall construct an element
$\Delta(G)$ from $\ZG_{2}$ as follows. For each crossing $c$ of $G$,
there are two ways of smoothing it. One way gives a knot, and the
other smoothing gives a $2$-component link $G_c$. We take the one
giving a $2$-component link and write

\begin{equation}
\Delta(G)=\sum_{c}G_{c}\in \ZG_{2}
\end{equation}

\begin{thm}
$\Delta(G)$ is a well defined mapping from $\ZG$ to $\ZGG_{2}$
\end{thm}

The proof is standard and follows Turaev's original idea. One should
consider the three Reidemeister move. The first move adds a new
summand which has a free loop (the latter assumed to be trivial in
$\ZGG_{2}$); for the second Reidemeister move we get two new
identical summands, which cancel each other because we are dealing
with $\Z_{2}$ coefficients. For the third Reidemeister moves the LHS
and the RHS will lead to the summands identical up to second
Reidemeister moves.

\begin{st} The free knot $K_1$ shown in Fig. \ref{frknot1} is
minimal.
\end{st}

\begin{figure}
\centering\includegraphics[width=200pt]{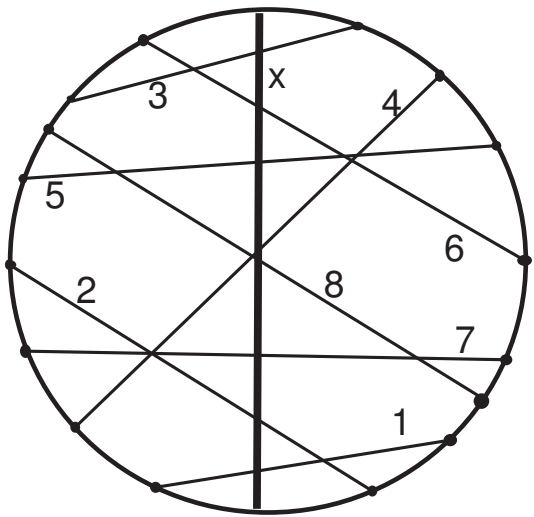} \caption{A
Non-trivial free knot with all even crossings} \label{frknot1}
\end{figure}
To see the minimality, one takes $\Delta(K_1)$ which, as an element
of $\ZG_{2}$, coincides with the link shown in Fig. \ref{frlnk1};
there is exactly one summand obtained by smoothing at $x$; all other
summands cancel by symmetry over $\Z_{2}$. Now, considering
$\{\Delta(K_{1})\}$, we get the desired claim.

Analogously one defines the maps $\Delta_{even}$ and $\Delta_{odd}$
corresponding to smoothing at even (resp., odd) crossings,
\cite{Ma}.

\section{Parity and groups}

It turns out that the parity argument can be used to construct some
extremely simple invariants of free knots; moreover, these
invariants turn out to be obstructions for a free knot to be slice.

The idea goes as follows. Consider the group $G=\langle
a,b,b'|a^{2}=b^{2}=b'^{2}=1,ab=b'a\rangle$. For a Gauss diagram $D$,
we shall construct a word $\gamma(D)$ in the alphabet $a,b,b'$ such
that the conjugacy class of this word in $G$ is an invariant of the
corresponding free knot. Let $D$ be an oriented chord diagram on a
circle $C$, and let $X$ be a fixed point on $C$ distinct from chord
ends.

We say that an odd chord $c$ of $D$ is {\em of the first type} if it
is linked with an odd number of {\em odd chords}, and {\em of the
second type} if it is linked with an odd number of {\em even}
chords.

With the pointed diagram $D$ we associate a word in the alphabet
$a,b,b'$ in the following way. We start walking along the
orientation of the circle $C$ from the point $X$. Every time when we
meet a chord end, we write a letter $a,b$, or $b'$ depending on
whether the corresponding chord is {\em even}, {\em first type odd}
or {\em second type odd}.

Denote the obtained word by $\gamma(D)$; this word generates an
element $[\gamma(D)]$ in the group $G$; sometimes we shall
abbreviate notation and denote $[\gamma(D)]$ just by $\gamma(D)$.

\begin{thm}
If pointed chord diagrams $(D,X)$ and $(D',X')$ generate equivalent
(oriented) free knots then  $[\gamma(D)]=[\gamma(D')]$ in
$G$.\label{inva}
\end{thm}

The proof of this theorem is a simple check of the invariance under
Reidemeister moves. Theorem \ref{inva} has an obvious
\begin{crl}
The conjugacy class of the element $[\gamma(D,X)]$ in $G$ is an
invariant of free knots.
\end{crl}
Indeed, moving the fixed point through a chord end, we change the
word cyclically, which means conjugation in $G$.

The group  $G$ admits a simple combinatorial description. Its Cayley
graph looks like a vertical strip on a squared paper between $x=0$
and $x=1$: we choose the point $(0,0)$ to be the unit in the group;
the multiplication by  $a$ on the right is chosen to one step in a
horizontal direction (to the right if the first coordinate of the
point is equal to zero, and to the left if this first coordinate is
equal to one), the multiplication by $b$ is one step upwards if the
sum of coordinates is even and one step downwards if this sum is
odd, and the multiplication by  $b'$ is one step downwards if the
sum of coordinates is even and one step upwards if the sum of
coordinates is even, see Fig. \ref{band}.

\begin{figure}
\centering\includegraphics[width=40pt]{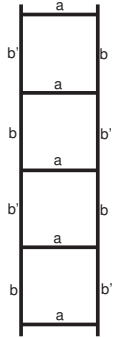} \caption{The Cayley
graph of the group $G$} \label{band}
\end{figure}

With each pointed chord diagram one associates an element from $G$
having coordinates $(0,4l)$. Moreover, the conjugacy class of the
element $(0,4l)$ for $l\neq 0$ consists of the two elements:
$(0,4l)$ and $(0,-4l)$. Thus, for each {\em long free knot} one gets
an integer-valued invariant, equal to  $l$; we shall denote this
invariant by  $l(K)$; each compact free knot has, in turn, the
invariant equal to  $|l|$; we shall denote the latter by $L(K)$.

In Fig. \ref{fig1} we depict a free knot $K_{1}$ with  $l(K_{1})=4$.

\begin{figure}
\centering\includegraphics[width=150pt]{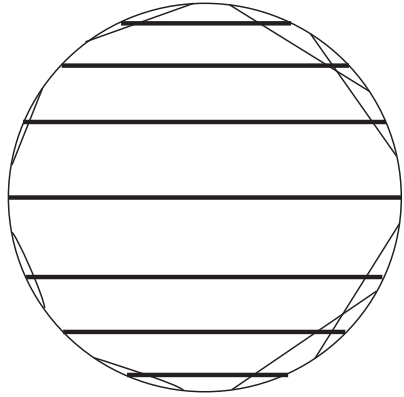} \caption{A non-slice
free knot} \label{fig1}
\end{figure}

So, it is possible to prove non-triviality of some free knots by
using parity arguments in a much simpler ways than in the initial
sections of the present paper.

In some consequent paper we shall prove that $L(K)$ is a {\em
cobordism invariant of free knots}, i.e., if $L(K)$ is non-zero then
the free knot $K$ is not null-cobordant.

This invariant was strengthened in \cite{MM}.


\begin{thebibliography}{100}

\bibitem{Af} D.M.Afanasiev, {\em On strengthening tnvariants of
virtual knots by using parity}, {\em Sbornik Mathematics}, to
appear.

\bibitem{Ca} J.S.Carter (1991), {\em Closed Curves that never extend to
proper maps of disks}, {\em Proc. Amer. Math. Soc.}, {\bf 113}, No.
3, pp. 879-888.

\bibitem{Fom} Fomenko A. T.
(1991), The theory of multidimensional integrable hamiltonian
systems (with arbitrary many degrees of freedom). Molecular table of
all integrable systems with two degrees of freedom, {\em Adv. Sov.
Math}, {\em 6}, pp.  1-35.

\bibitem{Gib} A.Gibson (2009) {\em Homotopy Invariants of Gauss
words}, Arxiv:Math.GT./0902.0062

\bibitem{Goldman} Goldman, W. M., Invariant functions on Lie groups and Hamiltonian
flows of surface group representations. {\em Invent. Math.} 85 , no.
2,(1986) 263-302

\bibitem{GPV} Goussarov M., Polyak M., and Viro O,
Finite type invariants of classical and virtual knots//  Topology.
2000. V. 39. P. 1045--1068.

\bibitem{IM} D.P.Ilyutko, V.O.Manturov, {\em Graph-links}, ibid.

 \bibitem{IM1}
D.\,P.~Ilyutko, V.\,O.~Manturov (2009),  Introduction to graph-link
theory, {\em Journal of Knot Theory and Its Ramifications}, {\bf 18}
(6), pp. 791-823.

\bibitem{IM2} D.P.Ilyutko, V.O.Manturov (2009), {\em Doklady Mathematics}, v. 428, {\bf 5},
pp. 591-594.


 \bibitem{KaV}
L.\,H.~Kauffman, Virtual knot theory, Eur. J. Combinatorics. 1999.
V.\ 20, N.\ 7, pp.\ 662--690.

 \bibitem{KauffmanBracket}
L.\,H.~Kauffman, State Models and the Jones Polynomial (1987), {\em
Topology}, {\bf 26}, pp.\ 395--407.


\bibitem{KM} D.Yu. Krylov, V.O.Manturov, Relative parity, {\em Work
in Progress}


\bibitem{MyBook} V.\,O.~Manturov (2004), {\em Knot Theory}, Champan and
Hall/CRC, Boca Raton, 416 pp.

\bibitem{Ma} V.O.Manturov, {\em Parity in Knot Theory}, {\em Sbornik
Mathematics}, to appear.

\bibitem{Ma1} V.O.Manturov, {\em On Free Knots},
ArXiv:Math.GT/0901.2214 v2

\bibitem{Ma2} V.O.Manturov, {\em On Free Knots and Links}, ArXiv:
Math.GT/0902.0127

\bibitem{Ma3} V.O.Manturov, {\em Free Knots are Not Invertible},
ArXiv:Math.GT/0909.2230


\bibitem{MM} V.O.Manturov, O.V.Manturov, {\em Free Knots and
Groups}, Arxiv:Math.GT/09122694, v.2.

\bibitem{Turaev} V.G.Turaev, Skein quantization of Poisson algebras of
loops on surfaces. {\em Ann. Sci. Ecole Norm. Sup.} (4) 24, no. 6,
(1991) 635–704.

\bibitem{TuProblem} V.G.Turaev,Topology of words, {\em Proc. Lond. Math. Soc.} (3) 95 (2007), no. 2, 360–412.

\bibitem{Turg} V.G.Turaev, A simple proof of the Murasugi and
Kauffman theorems on alternating links, {\em L'Enseignement
Math\'ematique.} {\bf 33} (1987), pp. 203- 225.

 \bibitem{Tu} V.G.Turaev, {\em Cobordisms of Words},
Arxiv:Math.CO/0511513, v.2.

\bibitem{Tu2} V.G.Turaev, {\em Cobordisms of Knots on Surfaces},
Arxiv:Math.GT/0703055, v.1.

\bibitem{VStrings} V.G, {\em Virtual open strings and their cobordisms}, preprint, 2004, arXiv:
math.GT/ 0311185 v5



\end{thebibliography}
\end{document}